\documentclass[11pt]{article}
\usepackage[english]{babel}
\usepackage{epsfig}
\usepackage{amsmath}
\usepackage{amsthm}
\usepackage{amssymb}
\usepackage{mathrsfs}
\usepackage[mathscr]{euscript}
\usepackage{color}

\newcommand{\ut}{\underline{t}}
\newcommand{\us}{\underline{s}}
\newcommand{\uta}{\underline{\tau}}

\newcommand{\ux}{\underline{x}}

\newtheorem{theorem}{Theorem}[section]
\newtheorem{definition}[theorem]{Definition}

\newtheorem{lemma}[theorem]{Lemma}
\newtheorem{proposition}[theorem]{Proposition}
\newtheorem{corollary}[theorem]{Corollary}
\newtheorem{remark}[theorem]{Remark}

\newtheorem{example}[theorem]{Example}

\newcommand{\rr}{{\mathbb{R}}}

\newcommand{\unx}{\underline{x}}
\newcommand{\unz}{\underline{z}}
\newcommand{\unu}{\underline{u}}
\newcommand{\utau}{\underline{\tau}}

\title{\bf On the Bargmann-Radon transform in the monogenic setting}

\author{
Fabrizio Colombo\\
Politecnico di Milano\\
Dipartimento di Matematica\\
Via Bonardi, 9\\
20133 Milano, Italy\\
fabrizio.colombo@polimi.it\and
Irene Sabadini\\
Politecnico di Milano\\
Dipartimento di Matematica\\
Via Bonardi, 9\\
20133 Milano, Italy\\
irene.sabadini@polimi.it\and
Franciscus Sommen\\
Clifford Research Group\\ Faculty of Sciences\\ Ghent University\\
Galglaan 2, 9000 Gent, Belgium\\ Franciscus.Sommen@UGent.be
}
\date{  }
\begin{document}
\maketitle
\begin{abstract}
In this paper we introduce and study a Bargmann-Radon transform on the real monogenic Bargmann module. This transform is defined as the projection of the real Bargmann module on the closed submodule of monogenic functions spanned by the monogenic plane waves. We prove that this projection can be written in integral form in terms the so-called Bargmann-Radon kernel. Moreover, we have a characterization formula for the Bargmann-Radon transform of a function in the real Bargmann module in terms of its complex extension and then its restriction to the nullcone in $\mathbb C^m$. We also show that the formula holds for the Szeg\H{o}-Radon transform that we introduced in \cite{css}. Finally, we define the dual transform and we provide an inversion formula.
\end{abstract}
{\bf Key words}: Monogenic functions,
Bargmann modules, Bargmann-Radon transform.

\noindent {\bf Mathematical Review Classification numbers}: 30G35, 44A12, 30H20.

\section{Introduction}
In the paper \cite{ppss} we considered an extension of
the Segal-Bargmann transform, a unitary map from spaces of square-integrable functions to spaces of square-integrable holomorphic functions (see \cite{Ba}, \cite{neretin}, \cite{Se1}, \cite{Se2}, \cite{zhu}).
Specifically, we studied the higher dimensional extension based on monogenic functions with values in a Clifford algebra. This approach has been used, e.g., in \cite{kmnq}, \cite{mnq} to study quantum systems with internal, discrete degrees of freedom corresponding to nonzero spins.

In \cite{ppss} we introduced a notion of Segal-Bargmann module (over the Clifford algebra) which is the set of entire functions, square integrable with respect to the Gaussian density and that are in the kernel of the Dirac operator. We also defined the Segal-Bargmann-Fock transform in this framework. The fact that monogenic functions admit a Fischer decomposition allows to prove a relation between the projection of the transform onto its monogenic part and the Fourier-Borel kernel. It is also worthwhile to mention that this kernel, unlike what happens for Hardy or Bargman spaces, is an exponential not a rational function.

In \cite{css} we defined the so-called Szeg\H{o}-Radon projection which may be abstractly defined as the orthogonal projection of a suitable Hilbert module of square integrable left monogenic functions onto the closed submodule of monogenic functions spanned by the monogenic plane waves $\langle\underline{x}, \utau\rangle^k\, \utau$, where $\utau =\underline{t}+i\underline{s}$, $\underline{t}$, $\underline{s}$ are orthogonal unit 1-vectors.  This transformation does not exactly correspond to the Radon transform. However it is a canonical map from $m$-dimensional monogenic functions to $2$-dimensional monogenic functions, like in the  case of the Clifford-Radon transform, see \cite{bls}, \cite{So1}. The Clifford-Radon transform and, more in general, the Radon transform are important tools with several applications for example in tomography.

In this paper we combine the approaches in \cite{css} and \cite{ppss}. We introduce and study a Bargmann-Radon transform on the real monogenic Bargmann module. Similarly to what we have done in \cite{css} in the Szeg\H{o}-Radon case, it is defined as the projection of the real Bargmann module on the closed submodule of monogenic functions spanned by the monogenic plane waves $\langle\underline{x}, \utau\rangle^k\, \utau$, where $\utau =\underline{t}+i\underline{s}$, $\underline{t}$, $\underline{s}$ are orthogonal unit 1-vectors. We show that this projection can be written in integral form in terms the so-called Bargmann-Radon kernel. A main result that we prove is a characterization formula for the Bargmann-Radon transform of a function in the real Bargmann module in terms of its complex extension and its restriction to the nullcone in $\mathbb C^m$. We also show that the same formula holds for the Szeg\H{o}-Radon transform treated in \cite{css}.
Finally, we study the dual Bargmann-Radon and as a by-product we obtain a formula, in integral form, to express the monogenic part of a holomorphic function belonging to the Bargmann module in several complex variables.
\\
The plan of the paper is the following. After the Introduction, Section 2 contains the notations and some preliminary results. In section 3 we introduce the real monogenic Bargmann module $\mathscr{BM}(\mathbb R^m)$ and we recall the definition of Segal-Bargmann-Fock space. We then define the Bargmann and the Bargmann-Radon transforms on $\mathscr{BM}(\mathbb R^m)$. We introduce the Bargmann-Radon kernel and we use to write the Bargmann-Radon transform in integral form. We conclude the section with a characterization formula. In Section 4, we recall the Szeg\H{o}-Radon transform, its associated kernel, and we show that the characterization formula holds also in this case. Finally, Section 5 contains the definition of dual transform and the inversion formula. These are similar to the analogue concepts in the Szeg\H{o}-Radon case treated in \cite{css}. We also obtain a  formula to write in integral form the monogenic part of a holomorphic function in several complex variables and, an an example,  we use it to express the Fourier-Borel and the Szeg\H{o} kernels.
\section{Notations and preliminary results}
In this section we collect some preliminary results and notations used in the rest of the paper.
For more information on the material in this section, we refer the reader to \cite{bds}, \cite{dss}.\\
By $\rr_m$ we denote the real Clifford algebra over $m$ imaginary units
$\underline{e}_1,\dots,\underline{e}_m$ which satisfy the relations $\underline{e}_i\underline{e}_j+\underline{e}_j\underline{e}_i=-2\delta_{ij}$.
An element $x$ in the
Clifford algebra is denoted by $x=\sum_A \underline{e}_Ax_A$ where $x_A\in\rr$,
$A=i_1\ldots i_r$, $i_\ell\in \{1,2,\ldots, n\}$, $i_1<\ldots <i_r$ is a multi--index,
$\underline{e}_A=\underline{e}_{i_1} \underline{e}_{i_2}\ldots \underline{e}_{i_r}$ and $\underline{e}_{\emptyset} =1$.
Similarly, we denote by $\mathbb C_m$ we denote the complex Clifford algebra over $m$ imaginary units
$\underline{e}_1,\dots,\underline{e}_m$

The so called 1-vectors are elements in $\mathbb{R}_m$ which are linear combinations with real coefficients of the elements $\underline{e}_i$, $i=1,\ldots, m$. The sets of 1-vectors is denoted by $\mathbb R_m^{(1)}$. The map from $\mathbb R^m$ to $\mathbb R_m^{(1)}$ is given by $(x_1,x_2,\ldots,x_m)\mapsto \underline{x}=
x_1\underline{e}_1+\ldots+x_m\underline{e}_m$ and it is obviously one-to-one. The norm of a 1-vector is defined as $|\unx|=(x_1^2+\cdots +x_n^2)^{1/2}$ and the scalar product of $\unx$ and $\underline{y}=y_1\underline{e}_1+\cdots +y_m\underline{e}_m$ is
$$\langle\unx,\underline{y}\rangle=x_1y_1+\cdots +x_my_m.$$
\\
In $\mathbb C_m$ there are automorphisms which leave the multivector structure invariant. In this paper we will use the so-called Hermitian conjugation
$$
(\lambda \mu)^\dagger = \mu^\dagger \lambda^\dagger,  \quad (\mu_A \underline{e}_A)^\dagger = \mu_A^c \underline{e}_A^\dagger, \quad \underline{e}_j^\dagger = - \underline{e}_j, \ \ \ \  j=1,\ldots,n,
$$
where $\mu_A^c$ stands for the complex conjugate of the complex number $\mu_A$.

In the sequel, we will denote by $B(0,1)$ the unit ball with center at the origin in $\mathbb R^m$
while the symbol $\mathbb{S}^{m-1}$ will denote its boundary, that is the sphere of unit 1-vectors in $\mathbb{R}^m$:
$$
\mathbb{S}^{m-1}=\{ \underline{x}=\underline{e}_1x_1+\ldots +\underline{e}_mx_m\ :\  x_1^2+\ldots +x_m^2=1\},
$$
whose area, denoted by $A_m$ is given by
$$
A_{m}=\frac{2 \pi^{m/2}}{\Gamma(\frac{m}{2})}.
$$

\begin{definition}
A function $f:\Omega\subseteq\mathbb R^m\to\mathbb C_m$ defined and continuously differentiable in the open set $\Omega$ is said to be (left) monogenic if it satisfies
$$
\partial_{\unx}f(\unx)=\sum_{j=1}^m \underline{e}_i\partial_{x_j} f(\ux)=0.
$$
If $f:\Omega\subseteq\mathbb C^m\to\mathbb C_m$ is as above, we say that $f$ is (left) monogenic in $\Omega$ if it is holomorphic and in the kernel of the complexified Dirac operator $\sum_{j=1}^m \underline{e}_i\partial_{z_j}$.
We denote by $\mathscr M(\Omega)$ the right $\mathbb C_m$-module of (left) monogenic functions in $\Omega$.
\end{definition}

A classical tool in Clifford analysis is the so-called Fischer decomposition. It provides a unique decomposition of an arbitrary homogeneous polynomial in $\mathbb R^m$ as
 $$
 R_k(\unx)=M_k(\unx)+\unx R_{k-1}(\unx),
 $$
 where the subscripts denote the degree of homogeneity of the polynomial and $M_k\in\mathscr{M}(\mathbb R^m)$.
 The monogenic polynomial $M_k$ is called monogenic part of $R_k$ and is denoted by $M(R_k)$.
The Fischer decomposition of the function $\frac{1}{k!}\langle\unx,\underline{u}\rangle^k$ can be written in terms of the so-called  zonal spherical monogenics which are defined by
\begin{equation}\label{zmk}
 Z_{k}(\underline{u},\underline{x})=\frac{\Gamma\left(\frac{m}{2}-1\right)}{2^{k+1}\Gamma\left(\frac{m}{2}+k\right)}
 (|\underline{u}| |\underline{x}|)^k
 \left((k+m-2)C_k^{\frac{m}{2}-1}(t)+(m-2) \frac{\underline{u}\wedge\underline{x}}{|\underline{u}| |\underline{x}|} C_{k-1}^{\frac{m}{2}}(t)\right)
 \end{equation}
where
  $t:=\dfrac{\langle\underline{u},\underline{x}\rangle}{|\underline{u}| |\underline{x}|}$ and
 $C^\lambda_k(t)$ are the Gegenbauer polynomials. Let us define
 $Z_{k,0}(\underline x,\underline u)=Z_{k}(\underline x,\underline u)$
and
\[Z_{k,s}(\underline x,\underline u)=\frac{Z_{k-s,0}(\ux,\underline u)}{\beta_{s,k-s}\ldots \beta_{1,k-s}},\quad k\ge s\]
with $\beta_{2s,k}=-2s$, $\beta_{2s+1,k}=-(2s+2k+m)$. Then we have:
\begin{equation}\label{FischerE}
\frac{1}{k!}\langle\unx,\underline{u}\rangle^k=\sum_{s=0}^k \unx^s Z_{k,s}(\unx,\unu )\unu^s.
\end{equation}
Using \eqref{FischerE}, we obtain the Fischer decomposition of $\exp(\langle \unz, \unx\rangle)$, see \cite{ndss}, namely
\[
\begin{split}
\exp(\langle \unx, \unu\rangle)&=\sum_{k=0}^\infty \frac{\langle \unx, \unu\rangle^k}{k!}\\
&=\sum_{k=0}^\infty\sum_{s=0}^k \unx^s Z_{k,s}(\unx,\unu )\unu^s\\
&=\sum_{s=0}^\infty\unx^s (\sum_{k=s}^\infty  Z_{k,s}(\unx,\unu ))\unu^s\\
&=  E(\unx,\unu)+\sum_{s=1}^\infty \unx^s E_s(\unx,\unu) \unu^s,
\end{split}
\]
where $E_s(\unx,\unu)= \sum_{k=s}^\infty Z_{k,s}(\unx,\unu)$. The function $E(\unx,\unu)$ is the monogenic part of $\exp(\langle \unx, \unu\rangle)$ and it is the Fourier-Borel kernel, see \cite{ndss}, \cite{ss}.
 Note that it is hermitian, namely, $E^\dagger(\unu, \unx)=E(\unx,\unu)$.

\section{The Bargmann-Radon transform}
In this section we introduce and study the Bargmann-Radon transform on the (real) monogenic Bargmann module. In particular, we introduce the Bargmann-Radon kernel and we use it to express the Bargmann-Radon transform in integral form. We also show that this transform gives rise to monogenic functions that can be expressed in an interesting way on the nullcone.\\
We begin by giving the definition of the so-called monogenic Bargmann module (see section 5 in \cite{ppss}):
\begin{definition}
The monogenic Bargmann module $\mathscr{MB} (\rr^m)$ consists of the functions $f\in\mathscr M(\rr^m)$ such that
$$
f(\unx)e^{-|\unx|^2/4}\in \mathscr{L}^2(\rr^m),
$$
and equipped with the inner product
$$
\langle f,g\rangle_{\mathscr{MB} }=\frac{1}{(2\pi)^{m/2}}\int_{\rr^m}  e^{-|\unx|^2/2} f^\dagger(\unx)g(\unx)\, d\unx .
$$
\end{definition}
Note that an analogous definition has been given in \cite{ppss}, section 4, for functions in the kernel of the complexified Dirac operator (or its powers). More precisely, we have
\begin{definition}
The Segal-Bargmann-Fock space $\mathscr B(\mathbb C^m)$ is the Hilbert space of entire functions in $\mathbb C^m$ which are square-integrable with respect to the $2m$-dimensional Gaussian density, i.e.
\[
\frac{1}{\pi^m}\int_{\mathbb C^m}\exp\left(-|\unz|^2\right) |f(\unz)|^2 \, d\unx d\underline{y}<\infty, \qquad \unz=\unx+i\underline{y}
\]
and equipped with the inner product
\[
\langle f,g\rangle_{\mathscr{B}} =\frac{1}{\pi^m}\int_{\mathbb C^m} \exp\left(-|\unz|^2\right){f^\dagger(\unz)}g(\unz)\, d\unx d\underline{y}.
\]
The monogenic Bargmann module $\mathscr{MB}(\mathbb C^m)$ is defined as
\[
\mathscr{MB}(\mathbb C^m)=\mathscr M (\mathbb C^m)\cap \mathscr B(\mathbb C^m),
\]
and it is is equipped with the inner product defined in $\mathscr B(\mathbb C^m)$.
\end{definition}
\begin{definition}
We define the Bargmann transform of $f\in\mathscr{MB} (\rr^m)$ as
\[
\mathcal B[f(\unx)e^{-|\unx|^2/4}]= (2\pi)^{-m/2}\int_{\rr^m} \exp( -\frac 12\langle \unz, \unz\rangle
+\langle \unx, \unz\rangle-\frac 14 |\unx|^2) f(\unx)e^{-|\unx|^2/4}\, d\unx .
\]
\end{definition}
We note that
\[
\begin{split}
\mathcal B[f(\unx)e^{-|\unx|^2/4}]&= (2\pi)^{-m/2}\int_{\rr^m} \exp( -\frac 12\langle \unz, \unz\rangle
+\langle \unx, \unz\rangle-\frac 14 |\unx|^2) f(\unx)e^{-|\unx|^2/4}\, d\unx\\
&=f(\unz)\\
&= (2\pi)^{-m/2}\int_{\rr^m} E(\unz,\unx) f(\unx) e^{-|\unx|^2/2}\, d\unx,
\end{split}
\]
where $E(\unz,\unx)$ is the monogenic part of $\exp(\langle \unz, \unx\rangle)$, see Section 2.

Let us denote by $[\cdot, \cdot]$  the Fischer inner product in $\mathscr M(\rr^m)$:
$$
[R,S]=R(\partial_{\unx})^\dagger S(\unx)_{|\unx = 0}.
$$
From this definition, we immediately obtain the formula
\begin{equation}\label{reproducingE}
\begin{split}
f(\unu)&=  (2\pi)^{-m/2}\int_{\rr^m} e^{-|\unx|^2/2} E(\unu,\unx ) f(\unx)\, d\unx \\
&= [E^\dagger (\unu,\unx), f(\unx)].
\end{split}
\end{equation}
By taking the holomorphic extensions of $f$, $g$ to $\mathbb C^m$ and using the fact that
$\mathscr{MB}(\mathbb R^m)$ equipped with the Fischer inner product and $\mathscr{MB}(\mathbb C^m)$ are isometric (see \cite{ppss}) we obtain
$$
\langle f ,g \rangle_{\mathscr{MB}} =\frac{1}{\pi^m}\int_{\mathbb C^m} e^{-|\unz|^2} f(\unz)^\dagger g(\unz)\, d\unz .
$$
Moreover, by the definition of Fischer product extended to functions defined over $\mathbb C^m$, we deduce
$$
f(\unu)= \frac{1}{\pi^m}\int_{\mathbb C^m} e^{-|\unz|^2} E(\unz,\unu)^\dagger f(\unz)\, d\unz .
$$
We now consider the following submodules of the module $\mathscr{MB}(\rr^m)$.
\begin{definition}\label{defsubmod}
For any given $\underline{\tau}= \underline{t}+i \underline{s}$, $\underline{t}$, $\underline{s}\in\rr^m$, where $|\underline{t}|=|\underline{s}|=1$, $\underline{t}\perp \underline{s}$, the closure of the right $\mathbb C_m$-module consisting of all finite linear combinations
$$
\sum_{\ell\in\mathbb N} \langle \unx, \utau\rangle^\ell \utau
$$
is denoted by $\mathscr{MB}(\utau)$.
\end{definition}
The following result from \cite{css} is useful for the computations in the sequel:
\begin{proposition}\label{proptt}
Let $\ut,\us\in\mathbb R^n$ be such that $|\ut|=|\us|=1$ and $\langle \ut,\us\rangle=0$ and let $\uta=\ut+i\us\in\mathbb C^m$.
Then
\begin{enumerate}
\item
$\uta\, \uta^\dagger\uta=4\uta$,
\item
 $\underline{\tau}^2=0$,
 \item
 $\underline{\tau}^\dagger \, \underline{\tau}+\underline{\tau}\,\underline{\tau}^\dagger=4$.
\end{enumerate}
\end{proposition}
We note that since $\utau^2=0$ then $\utau$ is an element in the nullcone in $\mathbb C^m$.

\begin{definition}
We define the Bargmann-Radon transform of $f\in \mathscr{MB}(\rr^m)$ as
$$
\mathcal R_{\utau}[f]={\rm Proj}_{\mathscr{MB}(\utau)} f
$$
where ${\rm Proj}_{\mathscr{MB}(\utau)}$ denotes the projection on $\mathscr{MB}(\utau)$.
\end{definition}
The Bargmann--Radon kernel is of the same form as the Szeg\H{o}-Radon kernel introduced in \cite{css} but with different coefficients:
$$
B_{\utau}(\unu,\unx)= \sum_{\ell=0}^\infty \lambda_\ell \langle \unu, \utau\rangle^\ell \utau\, \utau^\dagger \langle \unx, \utau^\dagger\rangle^\ell .
$$
\begin{remark}\label{coeff}{\rm
The calculations of the coefficients $\lambda_\ell$ follows from the fact that the integral
$$
\frac{1}{(2\pi )^{m/2}} \int_{\rr^m} e^{-|\unx|^2/2} \utau^\dagger\utau \langle \unx, \utau^\dagger\rangle^\ell \langle \unx, \utau\rangle^{\ell'}\, d\unx
$$
is zero when $\ell\not=\ell'$, because of the definition of the Fischer inner product. If $\ell=\ell'$ the integral equals
$$
\frac{(-1)^\ell}{(2\pi)^{m/2}} \utau^\dagger \utau \int_{\rr^m} e^{-|\unx|^2/2} | \langle \unx, \utau \rangle|^{2\ell}\, d\unx .
$$
This last integral does not depend on $\utau$ so we can compute it for a specific choice of $\utau$, for example $\utau= e_1+ie_2$. We have, by setting $\underline{y}=(x_3,\ldots , x_m)$,
\[
\begin{split}
\int_{\rr^m} e^{-|\unx|^2/2}  (x_1^2+x_2^2)^\ell d\unx&= \int_{\rr^{m-2}} e^{-|\underline{y}|^2/2} d\underline{y} \int_{0}^{+\infty}\int_{0}^{2\pi} e^{-r^2/2}r^{2\ell} r\, dr d\theta\\
&=  (2\pi)^{m/2}\int_{0}^{+\infty} 2^\ell e^{-s} s^{\ell} \, ds\\
&= (2\pi)^{m/2}2^{\ell} \ell !.
\end{split}
\]
}
\end{remark}
\begin{lemma}
The formula
\[
\frac{1}{(2\pi )^{m/2}} \int_{\rr^m} e^{-|\unx|^2/2} B_{\utau}(\unu,\unx) \utau  \langle \unx, \utau \rangle^\ell d\unx
= \utau \langle \unu, \utau \rangle^\ell
\]
holds if and only if
$$
\lambda_\ell =\frac{(-1)^\ell}{\ell!4\cdot 2^\ell}.
$$
Moreover
$$
B_{\utau}(\unu,\unx)= \sum_{\ell=0}^\infty \frac{(-1)^\ell}{\ell!4\cdot 2^\ell} \langle \unu, \utau\rangle^\ell \utau\, \utau^\dagger \langle \unx, \utau^\dagger\rangle^\ell =\frac{\utau\, \utau^\dagger}{4} \exp(-\frac 12 \langle \unu, \utau \rangle\langle \unx, \utau^\dagger \rangle).
$$
\end{lemma}
\begin{proof}
We compute
$$
\frac{1}{(2\pi )^{m/2}} \int_{\rr^m} e^{-|\unx|^2/2} B_{\utau}(\unu,\unx) \utau  \langle \unx, \utau \rangle^\ell d\unx
$$
using Remark \ref{coeff}. We have:
\[
\begin{split}
\frac{1}{(2\pi )^{m/2}} \int_{\rr^m} e^{-|\unx|^2/2} B_{\utau}(\unu,\unx) \utau  \langle \unx, \utau \rangle^\ell d\unx
&=\utau \, \utau^\dagger\utau \lambda_\ell (-1)^\ell \ell ! 2^\ell \langle \unu, \utau \rangle^\ell\\
&= \utau \langle \unu, \utau \rangle^\ell
\end{split}
\]
if and only if
$$
\utau \, \utau^\dagger\utau \lambda_\ell (-1)^\ell \ell ! 2^\ell
= \utau.
$$
By Proposition \ref{proptt} we have that $\utau \, \utau^\dagger\utau=4\utau$  and so we obtain the statement.
\end{proof}
Since $B_{\utau}$ is a reproducing kernel for the generators of $\mathscr{MB}(\utau)$, and since $\utau\,\utau^\dagger$ commutes with $\langle \unu, \utau \rangle$ we immediately have:
\begin{corollary}\label{cor25}
The function
$B_{\utau}(\unu,\unx)$
is a reproducing kernel for the $\mathbb C_m$-module $\mathscr{MB}(\utau)$.
\end{corollary}

The following result expresses the Bargmann-Radon transform of $f\in \mathscr{MB}(\rr^m)$ in terms of the Bargmann-Radon kernel:
\begin{theorem}\label{Th2.6}
Let $f\in \mathscr{MB}(\rr^m)$.
The following formula holds
$$
\mathcal{R}_{\utau}[f](\unu)=\frac{1}{(2\pi)^{m/2}}  \int_{\rr^m}
 e^{-|\unx|^2/2} B_{\utau}(\unu,\unx) f(\unx) d\unx .
$$
\end{theorem}
\begin{proof}
The assertion follows using standard arguments. First of all, we note that the operator $P$ defined by
\[
\begin{split}
P[f](\unu)&=\frac{1}{(2\pi)^{m/2}}  \int_{\rr^m}
 e^{-|\unx|^2/2} B_{\utau}(\unu,\unx) f(\unx) d\unx\\
 &=\frac{1}{(2\pi)^{m/2}} \frac{\utau\, \utau^\dagger}{4} \int_{\rr^m}
 e^{-|\unx|^2/2} \exp(-\frac 12 \langle \unu,\utau\rangle \langle \unx, \utau^\dagger\rangle) f(\unx) d\unx
 \end{split}
 \]
is idempotent on $\mathscr{MB}(\rr^m)$ and coincides with the identity on $\mathscr{MB}(\utau)$ by virtue of Corollary \ref{cor25}. The fact that the kernel $B_{\utau}(\unu,\unx)$ is hermitian gives $\langle Pf,g\rangle=\langle f,Pg\rangle$. Thus $P$ is the orthogonal projection of  $\mathscr{MB}(\rr^m)$ on $\mathscr{MB}({\utau})$ and thus it coincides with $\mathcal{R}_{\utau}$ as stated.
\end{proof}
Next result is interesting because it shows that the Bargmann-Radon transform of $f\in\mathscr{MB}(\rr^m)$ is a monogenic function which can be seen as a suitable multiple of the restriction to the nullcone of its extension to $\mathbb{C}^m$:
\begin{theorem}[Characterization formula]\label{BRtrans}
The Bargmann-Radon transform of $f\in\mathscr{MB}(\rr^m)$ is a monogenic function that can be expressed as:
$$
\mathcal R_{\utau}[f](\underline{u})=\frac{\utau\,\utau^\dagger}{4} f(-\frac 12 \utau^\dagger \langle \unu, \utau\rangle).
$$
\end{theorem}
\begin{proof}
First of all, any entire holomorphic function $h$ can be written as
$$
h(\unz)=M[h](\unz) +\unz g(\unz)
$$
where $M[h]$ denotes the monogenic part of $h$. Since $\utau^2=(\utau^\dagger)^2=0$ we have that
$$
\utau^\dagger h(\utau^\dagger)=\utau^\dagger M[h](\utau^\dagger).
$$
In particular, if we take $h(\unz)=\exp(-\frac{\lambda}{2} \langle \unx, \unz\rangle)$ we obtain:
\begin{equation}\label{emenolambda}
\begin{split}
\utau^\dagger \exp(-\frac{\lambda}{2} \langle \unx, \utau^\dagger\rangle)&= \utau^\dagger (E(\utau^\dagger ,-\frac{\lambda}{2}\unx)+
\utau^\dagger \ldots)\\
&=\utau^\dagger E(\utau^\dagger ,-\lambda/2 \unx)\\
&=\utau^\dagger E(-\lambda/2\utau^\dagger , \unx),
\end{split}
\end{equation}
thus, using \eqref{reproducingE}, we have:
$$
f(-\frac{\lambda}{2}\utau^\dagger)=  \frac{1}{(2\pi)^{m/2}}\int_{\rr^m} e^{-|\unx|^2/2} E(-\frac{\lambda}{2}\utau^\dagger,\unx ) f(\unx)\, d\unx .
$$
We now note that we can rewrite the formula in Theorem \ref{Th2.6} as
$$
\mathcal{R}_{\utau}[f](\unu)=\frac{1}{(2\pi)^{m/2}} \frac{\utau\, \utau^\dagger}{4} \int_{\rr^m}
 e^{-|\unx|^2/2} \exp(-\frac 12 \langle \unu,\utau\rangle \langle \unx, \utau^\dagger\rangle) f(\unx) d\unx
$$
$$
=\frac{1}{(2\pi)^{m/2}} \frac{\utau\, \utau^\dagger}{4} \int_{\rr^m}
 e^{-|\unx|^2/2} \exp(-\frac 12 \lambda \langle \unx, \utau^\dagger\rangle) f(\unx) d\unx_{|\lambda=\langle \unu,\utau\rangle}.
$$
Using \eqref{emenolambda} in the last formula, we get the statement.
\end{proof}

\section{The Szeg\H{o}-Radon transform}
In our paper \cite{css} we considered instead of the ambient module $\mathscr{MB}(\rr^m)$ another $\mathbb{C}_m$-module that we recall below:
\begin{definition}
The monogenic  Szeg\H{o} module is defined as the right $\mathbb{C}_m$-module $\mathscr{ML}^2(B(0,1))$ of the monogenic functions
$f:\ B(0,1)\subset \mathbb{C}^m\to \mathbb{C}_m$
for which $\lim_{r\to 1}f(r\underline{\omega})\in \mathscr{L}^2(\mathbb{S}^{m-1})$, equipped with
the Hilbert inner product
$$
\langle f,g\rangle_{\mathscr{ML}^2}=\int_{\mathbb{S}^{m-1}}f^\dagger(\underline{\omega})g(\underline{\omega}) dS(\underline{\omega}).
$$
\end{definition}
By the extension of the Cauchy formula, for $\underline{x}\in B(0,1)$, we have
\[
\begin{split}
f(\underline{x})&=\frac{1}{A_m}\int_{\mathbb{S}^{m-1}}\frac{\underline{x}-\underline{\omega}}{|\underline{x}-\underline{\omega}|^m}
\, \underline{\omega}\, f(\underline{\omega})\,dS(\underline{\omega})
\\
&=\frac{1}{A_m}\int_{\mathbb{S}^{m-1}}\frac{1+\underline{x}\underline{\omega}}{(1+|\underline{x}|^2-2\langle \, \underline{x},\underline{\omega}\rangle)^{m/2}}
\,f(\underline{\omega})\,dS(\underline{\omega})
\end{split}
\]
where  $A_m=\frac{2\pi^{m/2}}{\Gamma(m/2)}$.
For the standard Szeg\H{o}-Radon transform we again start from the plane waves
$$
f_{\utau, k}(\underline{x})=\langle\underline{x}, \utau\rangle^k\, \utau
$$
where $\utau=\underline{t}+i\underline{s}$ with $|\underline{t}|=|\underline{s}|=1$ and $\underline{t} \perp\underline{s}$ so that
$\utau \, \utau^\dagger+ \utau^\dagger\,\utau =4,$ see Proposition \ref{proptt}.
Since
$$
f^\dagger_{\utau, k}(\underline{x})=(-1)^k\langle\underline{x}, \utau^\dagger\rangle^k\, \utau^\dagger
$$
we obtain (see \cite{css}):
$$
\langle f_{\utau, k}, f_{\utau, k}\rangle= 2\pi^{m/2}\utau\, \utau^\dagger\frac{\Gamma(k+1)}{\Gamma(m/2+1)}.
$$
In the Szeg\H{o}-module we can give the analogue of Definition \ref{defsubmod}:
\begin{definition}
We denote by $\mathscr{ML}^2(\utau)$ the submodule of $\mathscr{ML}^2(B(0,1))$ which is the closure of the $\mathbb{C}_m$-module consisting of all finite linear combinations $\sum_k f_{\utau, k}(\underline{x})a_k$, $a_k\in\mathbb{C}_m$
of monogenic plane waves $f_{\utau, k}(\underline{x})$.
\end{definition}
As we have done in the previous section, we can consider the orthogonal projection on this submodule and we can describe its kernel, see  \cite{css}:
\begin{definition}
The Szeg\H{o}-Radon transform $\mathcal{R}_{\utau}[f]$ of $f\in \mathscr{ML}^2(B(0,1))$ is defined as the orthogonal projection of $f$ on the submodule $\mathscr{ML}^2(\utau)$.
\end{definition}
The kernel of this projection is given by
\[
\begin{split}
K(\underline{x},\underline{y})&=\frac{\utau\, \utau^\dagger}{4}\frac{\Gamma(m/2)}{2\pi^{m/2}}(1+\langle\underline{x}, \utau\rangle
\langle\underline{y}, \utau^\dagger\rangle)^{-m/2}
\\
&=\frac{\utau\, \utau^\dagger}{4}
\sum_{k=0}^\infty
\frac{\Gamma(m/2+k)}{2^{m/2}\Gamma(k+1)}\langle\underline{x}, \utau\rangle^k
\langle\underline{y}, \utau^\dagger\rangle^k
\end{split}
\]
and so we have
$$
\mathcal{R}_{\utau}[f]=\int_{\mathbb{S}^{m-1}}K_{\utau} (\underline{x},\underline{\omega})f(\underline{\omega}) dS(\underline{\omega}).
$$
The kernel $K_{\utau}$ can be directly related to the Szeg\H{o} kernel
$$
S(\underline{x},\underline{\omega})=\frac{1}{A_m}\frac{1+\underline{x}\underline{\omega}}{(1+|\underline{x}|^2-2\langle \, \underline{x},\underline{\omega}\rangle)^{m/2}}
$$
via the formula proved in the following lemma:
 \begin{lemma}\label{lemma1}
 We have
 $$
 K_{\utau}(\underline{x}, \underline{\omega})=\frac{\utau\,\utau^\dagger}{4}  S(-\langle \underline{x},\underline{\tau}\rangle \frac{\utau}{2},\underline{\omega})
 $$
 \end{lemma}
 \begin{proof}
 By setting $\lambda=\langle\underline{x},\utau\rangle$,  we obtain that
 $$
 S(-\frac{\lambda}{2}\utau^\dagger, \underline{\omega})=\frac{1}{A_m}
 \frac{1-\frac{\lambda}{2}\utau^\dagger \underline{\omega}}{(1+\lambda\langle\utau^\dagger, \underline{\omega}\rangle)^{m/2}}
 $$
 since $\langle \utau^\dagger, \utau^\dagger \rangle=0$, so the term that contains $|\underline{x}|^2$ disappear.
 So the formula follows from the fact that $\utau\, \utau^\dagger (\utau^\dagger \underline{\omega})=0$.
 \end{proof}
 The previous lemma allows to prove that the Szeg\H{o}-Radon transform satisfies the same characterization formula that we have obtained in Section 3 in the case of the Bargmann-Radon transform. This fact motivates the use of the same symbol $\mathcal{R}_{\utau}$ for both. Indeed we have:
 \begin{theorem} [Characterization formula]
Let $f\in \mathscr{ML}^2(B(0,1))$. Then the following formula holds:
 $$
 \mathcal R_{\utau}[f](\utau)=\frac{\utau\,\utau^\dagger}{4} f(-\frac 12 \utau^\dagger \langle \unu, \utau\rangle).
 $$
 \end{theorem}
 \begin{proof}
Lemma \ref{lemma1} implies directly the equalities
 \[
 \begin{split}
 \mathcal R_{\utau}[f](\utau)&=\int_{\mathbb{S}^{m-1}}K_{\utau}(\underline{x},\utau)f(\underline{\omega})dS(\underline{\omega})
 \\
& =\frac{\utau\,\utau^\dagger}{4} \int_{\mathbb{S}^{m-1}} S(-\frac{\utau^\dagger}{2}\langle \underline{x},\underline{\tau}\rangle , \underline{\omega})f(\underline{\omega})dS(\underline{\omega})
 \\
& =\frac{\utau\,\utau^\dagger}{4}f(-\frac{\utau^\dagger}{2}\langle \underline{x},\underline{\tau}\rangle),
 \end{split}
 \]
 and the statement follows.
 \end{proof}
 \begin{example}{\rm The following example is important because we consider the function
 $$
 g(\underline{x})=\sigma \langle \underline{x},\underline{\sigma}\rangle^\ell,\ \ \ \ \ \underline{ \sigma}^2=0,\ \ \ \ \ell\in \mathbb{N}
 $$
 which generates the module of all spherical monogenics of degree $\ell$. It can be verified with direct computations that
 $$
 \mathcal R_{\utau}[g](\utau)=\frac{\utau\,\utau^\dagger}{4} \underline{\sigma}
 \Big(-\frac{1}{2}\langle \utau^\dagger,\underline{\sigma}\rangle\langle \underline{x}, {\utau}\rangle\Big)^\ell .
 $$
 }
 \end{example}

 \section{The dual Bargmann-Radon transform and the inversion formula}
 Both the dual transform and the inversion formula will be the same for the Bargmann-Radon transform and for the Szeg\H{o}-Radon transform. The main results for the Szeg\H{o}-Radon transform were presented in \cite{css}. Here we repeat the results from the Bargmann-Radon point of view.

For every inner spherical monogenic $P_k(\underline{x})$ we have the formula
  $$
  P_k(\underline{u})=\frac{1}{(2\pi)^{m/2}}\int_{\mathbb{R}^m} e^{-|\underline{x}|^2/2} Z_{k}(\underline{u},\underline{x})P_k(\underline{x}) d\underline{x},
  $$
  which is in accordance with the fact that the Fourier-Borel kernel
  $$
  E(\underline{u},\underline{x})=\sum_{k=0}^\infty Z_{k}(\underline{u},\underline{x})
  $$
 is the reproducing kernel for the monogenic Bargmann module, see formula \eqref{reproducingE}.\\
 The Bargmann-Radon transform maps a monogenic function $f$ into $\mathscr{MB}(\mathbb{R}^m)$ into $\mathscr{MB}({\utau})$ and it can be expressed as, see Theorem \ref{BRtrans}:
 $$
 \mathcal R_{\utau}[f](\utau)=\frac{\utau\,\utau^\dagger}{4} f(-\frac 12 \utau^\dagger \langle \unu, \utau\rangle).
 $$
 \begin{definition}
 Let $F(\underline{u},\utau)$ be a function in $\mathscr{BM}(\utau)$. The dual
 Bargmann-Radon transform of $F$ is defined by
 $$
 \widetilde{\mathcal{R}}[F](\underline{u})=\frac{1}{A_mA_{m-1}}\int_{\mathbb{S}^{m-1}}\int_{\mathbb{S}^{m-2}}F(\underline{u},\underline{t}+i\underline{s})dS(\underline{t})dS(\underline{s})
 $$
 where for fixed $\underline{t}\in \mathbb{S}^{m-1}$ the  sphere $\mathbb{S}^{m-2}$ contains the elements $\underline{s}\in \mathbb S^{m-1}$ such that $\underline{s}\perp \underline{t}$.
 \end{definition}
 Note that the dual
 Bargmann-Radon transform is in fact the average of a function over the Stiefel manifold of 2-frames.
 \\
 Our main task in now to compute
 $\widetilde{\mathcal R}[\mathcal{R}_{\utau}f](\underline{u})$ and to relate this with $f$.
 As $f$ admits a monogenic Taylor series
 \begin{equation}\label{taylor}
 f(\underline{x})=\sum_{k=0}^\infty P_k(\underline{x})
 \end{equation}
 where $ P_k(\underline{x})$ are inner spherical monogenics of degree $k$, it will be sufficient to study
 $\widetilde{\mathcal{R}}[\mathcal{R}_{\utau}P_k](\underline{u})$, where
 \begin{equation}\label{Rtau}
 \mathcal{R}_{\utau}[P_k](\underline{u})=\frac{(-1)^k}{2^kk!}\frac{\utau\,\utau^\dagger}{4}
 \frac{1}{(2\pi)^{m/2}}\int_{\mathbb{R}^m} e^{-|\underline{x}|^2/2} (\langle\underline{u},\utau\rangle\langle\underline{x},\utau^\dagger \rangle  )^kP_k(\underline{x}) d\underline{x}.
 \end{equation}

 In our paper \cite{css}, see Theorem 5.4, we have proved a result which will be crucial in the sequel. We repeat it here for the sake of completeness and adapting the notation to the present setting:
 \begin{theorem}
For $\underline{\tau}=\underline{t}+i\underline{s}$, there exists a constant $\lambda_k$ such that
 \begin{equation}\label{formulalambda}
 \frac{1}{A_{m-1}A_{m}}
\int_{\mathbb{S}^{m-1}} dS(\underline{t})\int_{\mathbb{S}^{m-2}} dS(\underline{s})
\langle\underline{x}, \underline{\tau}\rangle^k\langle\underline{y}, \underline{\tau}^\dagger\rangle^k
\underline{\tau}\, \underline{\tau}^\dagger=
\end{equation}
$$
\lambda_k
(|\underline{u}| |\underline{x}|)^k
 \left((k+m-2)C_k^{\frac{m}{2}-1}(\frac{\langle\underline{u},\underline{x}\rangle}{|\underline{u}| |\underline{x}|})+(m-2) \frac{\underline{u}\wedge\underline{x}}{|\underline{u}| |\underline{x}|} C_{k-1}^{\frac{m}{2}}(\frac{\langle\underline{u},\underline{x}\rangle}{|\underline{u}| |\underline{x}|})\right)
$$
where for fixed $\underline{t}\in \mathbb{S}^{m-1}$ the  sphere $\mathbb{S}^{m-2}$ contains the elements $\underline{s}\in \mathbb S^{m-1}$ such that $\underline{s}\perp \underline{t}$
and  the constant $\lambda_k$ is given by
 $$
 \lambda_k= \frac{ 2\pi(-1)^k A_{m-2}}{(k+m-2)A_mC_k^{m/2-1}(1)}\, \frac{\Gamma\left(\frac{m}{2}-1\right)\Gamma\left(k+1\right)}{\Gamma\left(\frac{m}{2}+k\right)}.
 $$
 \end{theorem}
 We are now ready to compute
 $\widetilde{\mathcal R}[\mathcal{R}_{\utau} P_k](\underline{u})$:
 \begin{theorem}
 We have the relation
 $$
 \widetilde{\mathcal{R}}[\mathcal{R}_{\utau}P_k](\underline{u})=\frac{1}{2}\, \frac{\Gamma\left(m-1\right)\Gamma\left(k+1\right)}{\Gamma\left(m+k-1\right)}P_k(\underline{u}).
 $$
 \end{theorem}
 \begin{proof}
 In order  to compute the dual Bargmann-Radon transform of $\mathcal{R}_{\utau} P_k$, we need to compute  the integral
 $$
 I_k:=\frac{1}{A_mA_{m-1}}\int_{\mathbb{S}^{m-1}}\int_{\mathbb{S}^{m-2}}\langle\underline{u},\utau\rangle^k\langle
 \underline{x},\utau^\dagger
 \rangle^k \utau\,\utau^\dagger dS(\underline{t})dS(\underline{s}).
 $$
Using the fact that
 $$
C_k^{m/2-1}(1)= \frac{\Gamma\left(m-2+k\right)}{\Gamma\left(m-2\right)\Gamma\left(k+1\right)}
 $$
 we thus obtain that in fact
 $$
 I_k=\mu_kZ_{k}(\underline{u},\underline{x})
 $$
 with
 $$
 \mu_k=\frac{ 2\pi(-1)^k A_{m-2}}{A_m}\, \frac{\Gamma\left(m-2\right)\Gamma\left(k+1\right)}{\Gamma\left(m+k-1\right)}\, 2^{k+1}\Gamma\left(k+1\right)
 $$
 and $Z_{k}$ is the zonal spherical monogenic function defined in (\ref{zmk}).
 Since
 $$
 \frac{A_{m-2}}{A_m}=\frac{m-2}{2\pi}
 $$
 we have
 $$
  \mu_k=2^{k+1}(-1)^k \, \frac{\Gamma\left(m-1\right)\Gamma\left(k+1\right)^2}{\Gamma\left(m+k-1\right)}.
 $$
 Thus in order to compute the dual transform of $\mathcal{R}_{\utau}[P_k]$ we in fact have to compute
 $$
 \frac{(-1)^k}{2^{k+2}k!} I_k=\frac{1}{2}\, \frac{\Gamma\left(m-1\right)\Gamma\left(k+1\right)}{\Gamma\left(m+k-1\right)}Z_{k}(\underline{u},\underline{x})
 $$
We now have that
 $$
 \widetilde{\mathcal{R}}[\mathcal{R}_{\utau}P_k](\underline{u})=
 \frac{1}{(2\pi)^{m/2}}\int_{\mathbb{R}^m} e^{-|\underline{x}|^2/2} \widetilde{\mathcal{R}}[F](\underline{x})P_k(\underline{x}) d\underline{x}
 $$
 where
 $$
 \widetilde{\mathcal{R}}[F](\underline{x})=\frac{(-1)^k}{2^kk!}\frac{1}{4}\widetilde{\mathcal R}[\utau\,\utau^\dagger
 \langle\underline{u},\utau\rangle^k\langle\underline{x},\utau^\dagger\rangle^k]
 $$
 $$
 =\frac{1}{2}\, \frac{\Gamma\left(m-1\right)\Gamma\left(k+1\right)}{\Gamma\left(m+k-1\right)}Z_{k}(\underline{u},\underline{x})
 $$
 which, together with the reproducing  property of $Z_{k}$ leads to the result.
 \end{proof}
 Now we prove the following:
 \begin{theorem}[Inversion formula]
Let $f\in \mathscr{MB}(\mathbb{R}^m)$, then
$$
f(\underline{u})=\frac{2}{(m-2)!}(m-1-\Gamma)\ldots (1-\Gamma)\widetilde{\mathcal{R}}[\mathcal{R}_{\utau}f](\underline{u}),
$$
where
 $$
 \Gamma:=-\underline{x}\wedge \partial_{\underline{x}}
 $$
 is the $\Gamma$ operator.
 \end{theorem}
\begin{proof}
It is sufficient to decompose $f$  in Taylor series
$f(\underline{x})=\sum_{k=0}^\infty P_k(\underline{x})$,
to notice that $\Gamma P_k=-kP_k$ and to apply the previous result.
\end{proof}

 \begin{remark}{\rm
We observe that:
\begin{itemize}
\item[(i)] Monogenic functions satisfy the equation
$$
(E+\Gamma)f(\underline{x})=0
$$
where $E=\sum_{j=1}^nx_j\partial_{x_j}$ is the Euler operator.
So for monogenic functions we also have that
$$
f(\underline{u})=\frac{2}{(m-2)!}(m-1+E)\ldots (1+E)\tilde{R}[R_{\utau} f](\underline{u}).
$$
\item[(ii)] Notice that
$$
 \widetilde{\mathcal{R}}[\mathcal{R}_{\utau}f](\underline{u})=
 \frac{1}{A_mA_{m-1}}\int_{\mathbb{S}^{m-1}} \int_{\mathbb{S}^{m-2}}
 \frac{\utau\,\utau^\dagger}{4}
 f\left(-\frac{\utau^\dagger}{2} \langle\underline{u},\utau\rangle\right) dS(\underline{t}) dS(\underline{s})
 $$
 for which we have to use the complex extension $f(\underline{z})$ of $f(\underline{x})$ and put
 $$
 \underline{z}=-\frac{\utau^\dagger}{2} \langle\underline{u},\utau\rangle.
 $$
 \end{itemize}
 }
 \end{remark}

 An interesting consequence of this theory is an implicit formula for the monogenic part $M[h]$ of a holomorphic function belonging to the Bargmann module $\mathscr{B}(\mathbb{C}^m)$.
 Any entire holomorphic function $h$ admits the Fischer decomposition
 $$
 h(\underline{z})=\sum_{\ell=0}^\infty\underline{z}^\ell h_\ell(\underline{z})
 $$
 where $h_\ell$ is complex monogenic, that is
 $\partial_{\underline{z}}h_\ell(\underline{z})=0$.
Using  this fact we can prove:
\begin{theorem}\label{monogenicpart}
  The monogenic part $M[h]$ of an entire holomorphic function $h$ is given by (for $\underline{u} \in \mathbb{C}^m$ or in $\mathbb{R}^m)$
  $$
 M[h](\underline{u})=\Theta
 \frac{1}{A_mA_{m-1}}\int_{\mathbb{S}^{m-1}} \int_{\mathbb{S}^{m-2}}
 \frac{\utau\,\utau^\dagger}{4}
h\left(-\frac{\utau^\dagger}{2} \langle\underline{u},\utau\rangle\right) dS(\underline{t}) dS(\underline{s})
 $$
 where
 $$
 \Theta:=\frac{2}{(m-2)!}(m-1-\Gamma)(m-2-\Gamma)\ldots (1-\Gamma).
 $$
\end{theorem}
\begin{proof}
  The result hold in the case when
  $$
  f(\underline{u})=M[h](\underline{u})
  $$
  is the monogenic Bargmann module, i.e., when $h\in \mathscr{B}(\mathbb{C}^m)$ because indeed
  $$
  f(\underline{u})=\Theta
 \frac{1}{A_mA_{m-1}}\int_{\mathbb{S}^{m-1}} \int_{\mathbb{S}^{m-2}}
 \frac{\utau\,\utau^\dagger}{4}
f\left(-\frac{\utau^\dagger}{2} \langle\underline{u},\utau\rangle\right) dS(\underline{t}) dS(\underline{s}).
$$
Now, using the Fischer decomposition $
 h(\underline{z})=\sum_{\ell=0}^\infty\underline{z}^\ell h_\ell(\underline{z})
 $ of $h$ we obtain for
 $
 \underline{z}=-\frac{\utau^\dagger}{2} \langle\underline{u},\utau\rangle
 $
 $$
 \frac{\utau\,\utau^\dagger}{4}
h\left(-\frac{\utau^\dagger}{2} \langle\underline{u},\utau\rangle\right)=\frac{\utau\,\utau^\dagger}{4}\sum_{\ell=0}^\infty\underline{z}^\ell h_\ell(\underline{z})
 $$
$$
=\frac{\utau\,\utau^\dagger}{4} h_0(-\frac{\utau^\dagger}{2} \langle\underline{u},\utau\rangle)
$$
 and $h_0=M[h]=f$.
The result extends to holomorphic functions in a neighborhood of the origin.
\end{proof}
To our knowledge this is the first time that the monogenic part of a function is given in integral form and, as a possible application, we provide two examples namely the monogenic part of the Fourier-Borel kernel and of the Szeg\H{o} kernel.
\begin{example} {\rm The Fourier-Borel kernel $E(\unu,\unx)$.\\
As we explained at the end of Section 2, the Fourier-Borel kernel is the monogenic part of the function $\exp(\langle\unz,\unx\rangle)$. Applying Theorem \ref{monogenicpart}, where we set $h(\unz)=\exp(\langle\unz,\unx\rangle)$ so that $$h\left(-\frac{\utau^\dagger}{2}\langle\unu,\utau\rangle\right)=\exp\left(\langle-\frac{\utau^\dagger}{2}\langle\unu,\utau\rangle, \unx\rangle\right)=\exp\left(-\frac{1}{2}\langle \unx, \utau^\dagger\rangle\, \langle\unu,\utau\rangle\right),$$ we have
$$
E(\unu,\unx)=\Theta
 \frac{1}{A_mA_{m-1}}\int_{\mathbb{S}^{m-1}} \int_{\mathbb{S}^{m-2}}
 \frac{\utau\,\utau^\dagger}{4}
\exp\left(-\frac{1}{2} \langle\underline{x},\utau^\dagger\rangle \langle\underline{u},\utau \rangle\right) dS(\underline{t}) dS(\underline{s}).
 $$
 This formula expresses the Fourier-Borel kernel in terms of the integral
 $$
 H(\unu,\unx)=
 \frac{1}{A_mA_{m-1}}\int_{\mathbb{S}^{m-1}} \int_{\mathbb{S}^{m-2}}
 \frac{\utau\,\utau^\dagger}{4}
 \exp\left(-\frac{1}{2}  \langle\underline{u},\utau \rangle \langle\underline{x},\utau^\dagger\rangle \right) dS(\underline{t}) dS(\underline{s})
 $$
 that is a zonal biregular function.}
\end{example}
\begin{example} {\rm The Szeg\H{o} kernel $S(\unu,\unx)$.\\
We recall that
$$
S(\underline{z},\underline{x})=\frac{1}{A_m}\frac{1+\underline{z}\underline{x}}{(1+\langle \, \underline{z},\underline{z}\rangle|\underline{x}|^2-2\langle \, \underline{z},\underline{x}\rangle)^{m/2}}.
$$
Thus, using Theorem \ref{monogenicpart}, we obtain
$$
 S(\underline{u},\underline{x})=
 \frac{\Theta}{A_m^2 A_{m-1}}\int_{\mathbb{S}^{m-1}} \int_{\mathbb{S}^{m-2}}
 \frac{\utau\,\utau^\dagger}{2}(1+\langle \, \underline{u},\underline{\tau}\rangle\, \langle \, \underline{x},\underline{\tau}^\dagger\rangle)^{-m/2}\ dS(\underline{t}) dS(\underline{s})
 $$
which is also the monogenic part of the Cauchy-Hua kernel
$$
(1+ \langle \, \underline{z},\underline{z}\rangle \, \langle \, \underline{w}^\dagger,\underline{w}^\dagger\rangle +2 \langle \, \underline{z},\underline{w}^\dagger\rangle)^{-m/2}, \qquad \underline{w}^\dagger=- \underline{x},
$$
see also \cite{morimoto}, \cite{so3}.
}
\end{example}


\begin{thebibliography}{99}
\bibitem{Ba} V. Bargmann, On a Hilbert space of analytic functions and an associated integral
transform, Part I, Comm. Pure Appl. Math. 14 (1961), 187--214.

\bibitem{bds} F. Brackx, R. Delanghe, F. Sommen, {\em Clifford Analysis},
Pitman Res. Notes in Math., 76, 1982.

\bibitem{bls} J. Bure\v{s}, R. L\'avi\v cka,  V. Sou\v{c}ek, Elements of quaternionic analysis and Radon transform, Textos de Matematica 42, Departamento de Matematica, Universidade de Coimbra, Coimbra, 2009.

\bibitem{css} F. Colombo, I. Sabadini, F. Sommen, {\em On the Szeg\H{o}-Radon projection of monogenic functions}, Adv. Appl. Math., {\bf 74} (2016), 1--22.


\bibitem{dss} R. Delanghe, F. Sommen, V. Sou\v{c}ek, Clifford Algebra and
Spinor-valued Functions, Mathematics and Its Applications 53, Kluwer Academic Publishers Group, Dordrecht, 1992.

\bibitem{ndss} N. De Schepper, F. Sommen, {\em Closed form of the {F}ourier-{B}orel kernel in the framework
              of {C}lifford analysis}, {Results Math.}, {\bf 62} (2012), 181--202.

\bibitem{kmnq}  W. D. Kirwin, J. Mourao, J. P. Nunes, T. Qian, {\em Extending coherent state
transforms to Clifford analysis}, arXiv:1601.01380.

 \bibitem{morimoto} M. Morimoto, {\em Analytic functionals on the Lie sphere}, Tokyo J. Math., {\bf 3} (1980),  1--35.


       \bibitem{mnq} J. Mourao, J. P. Nunes, T. Qian, {\em Coherent State Transforms and the
Weyl Equation in Clifford Analysis}, arXiv:1607.06233.


\bibitem{neretin} Y. A. Neretin, {\em Lectures on Gaussian Integral Operators and Classical Groups}, Series of Lectures in Mathematics, European Mathematical Society, 2011.

\bibitem{ppss} D. Pena Pena, I. Sabadini, F. Sommen, {\em Segal-Bargmann-Fock modules of monogenic functions}, arxiv:1608.06790v1.

\bibitem{ss} I. Sabadini, F. Sommen, {\em A Fourier--Borel transform for monogenic functionals on the Lie ball},
 J. Math. Soc. Japan, {\bf 68} (2016), 1487--1504.

    \bibitem{Se1} I. Segal, {\em Mathematical characterization of the physical vacuum for a linear
Bose-Einstein field}, Illinois J. Math. 6 (1962), 500--523.

\bibitem{Se2} I. Segal, {\em The complex wave representation of the free Boson field}, in “Topics in
functional analysis: Essays dedicated to M.G. Krein on the occasion of his 70th
birthday, I. Gohberg and M. Kac, Eds., Advances in Mathematics Supplementary
Studies, Vol. 3, 321--343. Academic Press, New York, 1978.

\bibitem{So1} F. Sommen, An extension of the Radon transform to Clifford analysis, Complex Variables Theory Appl.
 8 (1987) 243-266.

\bibitem{so3} F. Sommen, {\em Spherical monogenic on the Lie sphere}, J. Funct. Anal. {\bf 92}, 372--402 (1990).


\bibitem{zhu} K. Zhu, {\em Analysis on Fock spaces}, Graduate Texts in Mathematics 263, Springer New York 2012.

\end{thebibliography}
\end{document}